\documentclass[a4paper,11pt,leqno]{article}	
\usepackage{graphicx}
\usepackage{relsize}
\usepackage{hyperref}
\usepackage{comment}

\medskipamount
\usepackage[usenames,dvipsnames]{color}

\usepackage{amsmath}
\usepackage{amsthm}
\usepackage{amssymb}
\usepackage{amscd}
\usepackage{enumerate}
\usepackage{pstricks}
\usepackage{pst-node}
\usepackage{esint}
\usepackage{mathtools}

\theoremstyle{plain}
\newtheorem{theorem}{Theorem}[section]

\newtheorem{lemma}[theorem]{Lemma}
\newtheorem{proposition}[theorem]{Proposition}

\theoremstyle{definition}
\newtheorem{assumption}[theorem]{Assumption}
\newtheorem{definition}[theorem]{Definition}

\theoremstyle{remark}
\newtheorem{remark}{Remark}

\newtheorem{example}{Example}

\newcommand\R{\mathbb{R}}
\newcommand\M{\mathbb{M}}
\newcommand\N{\mathbb{N}}

\newcommand\D{D}

\renewcommand\boldsymbol{}

\newcommand\calW{\mathcal{W}}

\newcommand{\SO}[1]{\operatorname{SO}(#1)}
\newcommand\id{Id}
\newcommand\sym{\operatorname{sym}}

\newcommand\esssup{\mathop{\operatorname{ess\,sup}}}

 \newcommand\dist{\operatorname{dist}}
\newcommand{\e}{\varepsilon}

\newcommand\eps{\varepsilon}
\newcommand\diver{\mathop{\operatorname{div}}}
\newcommand\argmin{\operatorname{argmin}}

\DeclareMathOperator{\curl}{curl}

\title{Homogenization of bending theory for plates; the case of elastic laminates}
\date{\today}

\author{Maroje Marohni\'{c},\\
University of Zagreb, Faculty of Natural Sciences and Mathematics,\\ Department of Mathematics,\\ Bijeni\v{c}ka 30, 10000 Zagreb, Croatia,\\
maroje.marohnic@math.hr, \\  \\
Igor Vel\v{c}i\'{c},\\
University of Zagreb, Faculty of Electrical Engineering and Computing,\\ Unska 3, 10000 Zagreb, Croatia,\\
igor.velcic@fer.hr}

\begin{document}

\maketitle

\begin{abstract}
In this paper we study the homogenization effects on the model of elastic plate in the bending regime, under the assumption that the  energy density (material) oscillates in the direction of thickness.
We study two different cases. First, we show, starting from 3D elasticity, by means of $\Gamma$-convergence
and under general (not necessarily periodic) assumption, that the effective behavior of
the limit is not influenced by  oscillations in the direction of thickness. In the second case, we study  periodic in-plane oscillations of the energy density coupled with periodic oscillations
in the direction of thickness. In contrast to the first case we show that there are
homogenization effects coming also from the oscillations in the direction of thickness.

\end{abstract}

\vspace{10pt}



\section{Introduction} \label{sectionprvi}

There is vast literature on the derivation of plate models from $3D$ elasticity by means of  $\Gamma$-convergence. The first work in this direction was \cite{LDRa95} where the authors derived the plate model in the membrane regime. It  is well known that different models of thin structures can be obtained depending on  the relation between the
magnitude of external loads (i.e., the energy) and the thickness of the body $h$.  The first work in deriving higher ordered models of plates was \cite{FJM-02} where the authors derived the plate model in the bending regime (order of the energy $h^2$, after dividing with the order of volume $h$). The key mathematical  ingredient was the theorem on geometric rigidity, and  by using this theorem the authors in \cite{FJM-06} also derived the models in the regimes higher in hierarchy than bending (von K\'arman with the order of the energy $h^4$, constrained models with the order between $h^4$ and $h^2$).  We also mention the work \cite{Schmidt-07}, relevant for this paper, where the author derived the plate model in the bending regime, under the assumption that the material is layered in the direction of thickness, but with fixed energy density.

In \cite{Braides-Fonseca-Francfort-00} the authors study the effects of simultaneous homogenization and dimensional reduction, by variational techniques, on the model of membrane plate. Their approach is also general and allow non-periodic oscillations. Moreover,  the boundary can also be oscillating.

 The influence of the effects of simultaneous homogenization and dimensional reduction on limit equations is studied in different contexts (see \cite{Arrieta96} for the Laplace equation on thin domain with periodically oscillating boundary,   \cite{JurakTutek} for the model of homogenized rod in the context of linear elasticity where the material periodically oscillates in the direction of thickness, \cite{Courilleau04} for nonlinear monotone operators without periodicity assumptions on the coefficients). In \cite{gustafsson06} the authors obtained the limit equations for the linear plate, where   the material oscillates in the direction of thickness (this is done without the periodicity assumption from $3D$ linear elasticity equations by combining $H$-convergence and dimension reduction techniques).

The effects of simultaneous homogenization and dimensional reduction on the  models of plates, shells and rods in different (nonlinear) higher ordered regimes  have recently been obtained. In \cite{Neukamm-11} the author derived the limit equations for the rod in the bending regime by combining two-scale convergence techniques together with the dimensional reduction techniques from \cite{FJM-02}.  Different models were obtained depending on the limit of the ratio between the periodicity of the oscillations  and the thickness  $h$.
 The oscillations of the material were assumed to be periodic along the central line of the rod.
 Together with \cite{Velcic-12} this is the first work where the higher ordered nonlinear models were derived combining the effects of homogenization and dimensional reduction.
 In \cite{NeuVel-12} the authors derived the plate model in von K\'arm\'an regime, assuming periodic in-plane oscillations of the material. In \cite{Horneuvel12} and \cite{Vel13} the authors derived the plate models in the bending regime   again depending  on the relation between the periodicity of  oscillations  and the thickness  $h$. The models of homogenized von K\'arm\'an shell are discussed in \cite{Hornungvel12}.
 In \cite{NeuOlb-13} the authors homogenize  the bending plate model ($2D$) with periodic oscillations. What is interesting here  is that the homogenization is done under the constraint of being an isometry.
 In \cite{Velcic-13} and \cite{MarVel-14} the authors study the effects of simultaneous homogenization and dimensional  reduction  for the plate model in von K\'arm\'an regime and the rod in the bending regime  without any assumptions on the periodicity or  the direction of the oscillations. These works thus generalize  \cite{NeuVel-12}  and \cite{Neukamm-11} respectively. The  derivation of the general plate model in the bending regime   without periodicity assumption has not been done, since the models, even in some periodic regimes, are unknown (see the discussion in  \cite{Vel13}  and also the model obtained in \cite{NeuOlb-13}).

 In this paper we study the effects of simultaneous homogenization and dimensional reduction to obtain the model of the plate in the bending regime. Unlike in \cite{Horneuvel12,Vel13} we suppose that the material  changes in the direction of thickness and we do not work under the periodicity assumption in the case when the material does not change in the in-plane directions rapidly. We show that the limit model in this case is simple, i.e., there are no homogenization effects (see Remark \ref{napomena1}). In the case when we couple periodic in-plane oscillations of the material together with the periodic oscillations in the direction of thickness, the effect of oscillations in the direction of thickness on the limit model is nontrivial and these effects are coupled together with the effects of oscillations in the in-plane direction and can not be separated (see Remark \ref{napomena2}). This is useful information for the structural optimization.

\subsection{Notation}
\begin{itemize}
\item $x'=(x_1,x_2)$, $x=(x_1,x_2,x_3)$
\item  $\nabla_h y:= (\nabla'y,\frac{1}{h}\partial_3{y} )$  is a scaled gradient and $\nabla'y = (\partial_1{y},\partial_2{y})$
\item $\mathbb{S}_n$ space of symmetric matrices of order $n$
\item weak $\rightharpoonup$ and strong $\to$ convergence
\item $\bar{y}(x')=\int_{I} y(x',x_3) dx_3$, $I=[-\tfrac{1}{2},\tfrac{1}{2}]$.
\item by $e_1, \dots,e_n$ we denote the vectors of the canonical base in $\R^n$.
\item we suppose that the greek indices $\alpha,\beta$ take the values in the set $\{1,2\}$, while the latin indices $i,j,k$  take the values in the set $\{1,2,3\}$, unless otherwise stated.

\item $\iota:\R^{2 \times 2} \to \R^{3 \times 3}$ denotes the natural embedding
$$ \iota(A)=\sum_{\alpha,\beta=1,2} A_{\alpha \beta} e_{\alpha} \otimes e_{\beta}.$$
\item by $A^t$ we denote the transpose of the matrix $A$.
\item by $\id$ we denote the identity matrix.
\end{itemize}

\section{Derivation of the model} \label{sectiondrugi}

Let $S \subset \R^2$ be an open connected set with Lipschitz boundary and also piecewise $C^1$. The property piecewise $C^1$ is necessary only for proving the upper bound.
We define by $\Omega^h =S \times\left(-\frac{h}{2},\frac{h}{2} \right)$ the reference configuration of the plate-like body. When $h=1$ we omit the superscript and simply write $\Omega=\Omega^1$.

For every $h>0$ we define the energy functional of elastic energy on a canonical domain $\Omega$
\begin{equation*}
I^h(y^h) := \int_{\Omega} W^h(x_3,\nabla_h y^h) dx,
\end{equation*}
where $W^h: I \times \R^{3 \times 3} \to[0,+\infty]$ denotes the stored energy density with the properties given below. We assume that we are in the bending regime, i.e., that there is a positive constant $C>0$, independent of $h$, such that:
\begin{equation}\label{bendingregime}
 I^h(y^h) \leq C h^2.
\end{equation}

\subsection{General framework}
The following two definitions will give conditions on the energy densities.
\begin{definition}[Nonlinear material law]\label{def:materialclass}
  Let $\eta_1,\eta_2,\rho$ be positive real constants such that
  $\eta_1\leq\eta_2$. We denote by  $\calW(\eta_1,\eta_2,\rho)$ the class of all measurable functions $W:\R^{3 \times 3}\to[0,+\infty]$ satisfying the following properties:
  \begin{align}
    \tag{W1}\label{ass:frame-indifference}
    &\text{frame-indifference:}\\
    &\notag\qquad W(R \boldsymbol F)=W(\boldsymbol F)\quad\text{ for
      all $\boldsymbol F\in\M^3$ and  $\boldsymbol R\in\SO
      3$;}\\
    \tag{W2}\label{ass:non-degenerate}
    &\text{non degeneracy:}\\
    &\notag\qquad W(\boldsymbol F)\geq \eta_1 \dist^2(\boldsymbol F,\SO 3)\quad\text{ for all
      $\boldsymbol F\in\M^3$,}\\
    &\notag\qquad W(\boldsymbol F)\leq \eta_2\dist^2(\boldsymbol F,\SO 3)\quad\text{ for all
      $\boldsymbol F\in\M^3$ with $\dist^2(\boldsymbol F,\SO 3)\leq\rho$;}\\
    \tag{W3}\label{ass:stressfree}
    &W\text{ is minimal at $\id$:}\\
    &\notag\qquad W(\id)=0;\\
    \tag{W4}\label{ass:expansion}
    &\text{W admits a quadratic expansion at $\id$:}\\
    &\notag\qquad W(\id+\boldsymbol G)=Q(\boldsymbol G)+o(|\boldsymbol G|^2),\qquad\text{for all }\boldsymbol G\in\M^3,\\
    &\notag\text{where $Q\,:\,\M^3\to\R$ is a quadratic form.}
 \end{align}
\end{definition}
\begin{definition}[Admissible composite material]\label{def:composite}
  Let $0<\eta_1\leq\eta_2$ and $\rho>0$. A family $(W^h)_{h>0}$
  \begin{equation*}
    W^h:I \times \R^{3 \times 3} \to [0,+\infty],
  \end{equation*}
  describes an admissible composite material of class $\mathcal W(\eta_1,\eta_2,\rho)$ if
  \begin{enumerate}[(i)]
  \item  For each $h>0$, $W^h$ is almost surely equal to a Borel function on $I\times\R^{3\times 3}$;
  \item $W^h(x_3,\cdot)\in\mathcal W(\eta_1,\eta_2,\rho)$ for  every $h>0$ and almost every $x\in\Omega$;
  \item there exists a monotone function $r:\R^+\to [0,+\infty]$, such that $\lim_{\delta \to 0} r(\delta)= 0$ and
  \begin{equation}\label{assumQh}
   (\forall h>0) \ (\forall \boldsymbol G\in\R^{3\times 3}) \quad \esssup_{x_3\in I} |W^h(x_3,\id+\boldsymbol G)-Q^h(x_3,\boldsymbol G)|\leq r(|\boldsymbol G|) |\boldsymbol G|^2,
  \end{equation}
where the family of quadratic forms $Q^h(x_3,\cdot)$ is as  in Definition \ref{def:materialclass}.
\end{enumerate}
\end{definition}

Note that  $Q^h$ inherits the measurability properties of $W^h$ since, for each $h>0$, it  can be written as the pointwise limit
\begin{equation} \label{defQ} (x_3,G) \to Q^h(x_3,G) := \lim_{\e \to 0}
\tfrac{1}{\e^2} W^h(x, Id+\e G).
\end{equation}

It is straightforward to establish the following properties of $Q^h$ from (W2).
\begin{lemma}
  \label{lem:111}
Suppose that $(W^h)_{h>0}$ describes an admissible composite material in the sense of the Definition~\ref{def:composite} and let $(Q^h)_{h>0}$ be the family of the quadratic forms
associated to $(W^h)_{h>0}$ through the expansion \eqref{ass:expansion}. Then the family $(Q^h)_{h>0}$ satisfies the following properties
  \begin{enumerate}[(i)]
  \item[(Q1)] for all $h>0$ and almost all $x_3\in I $ the map $Q^h(x_3,\cdot)$ is quadratic and satisfies
    \begin{equation*}
      \eta_1|\sym \boldsymbol G|^2\leq Q^h (x_3,\boldsymbol G)=Q^h(x_3,\sym \boldsymbol G)\leq \eta_2|\sym \boldsymbol G|^2,\qquad\text{for all $ \boldsymbol G\in\R^{3 \times 3}$.}
    \end{equation*}
  \item[(Q2)]
For all $h>0$ and almost all  $x_3 \in I$ the inequality
   \begin{equation*}
|Q^h(x_3,G_1)-Q^h(x_3,G_2)| \leq \eta_2 |\sym G_1-\sym G_2|\cdot |\sym G_1+\sym G_2|,
\end{equation*}
holds for all $ \boldsymbol G_1, \boldsymbol G_2 \in\R^{3 \times 3}$.
  \end{enumerate}
\end{lemma}

For each $h>0$ we define quadratic forms $Q_2^h:I \times\mathbb{R}^{2 \times 2} \times \mathbb{R}^{2 \times 2} \to \R$  by minimizing with respect to the third row and column:
\begin{equation}\label{defQ2}
Q^h_2(x_3,A,B)=\min_{d \in \mathbb{R}^3} Q^h(x_3,\iota(B+x_3A)+ d \otimes e_3 + e_3 \otimes d).
\end{equation}

Since by (Q1) the coefficients of $Q_2^h$ are uniformly bounded in $L^{\infty}(S)$
 we deduce that there is a subsequence $h_k$  such that the coefficients of $ Q_2^{h_k}$ converge in weak * topology in $L^{\infty}(I)$. We extend this property to whole sequence $(Q^h)_{h>0}$ by making the following assumption.
\begin{assumption} \label{ass:main}
We suppose that the coefficients of the quadratic forms $Q_2^h$ weakly converge as $h$ goes to zero.
\end{assumption}
By $Q_2$ we denote the quadratic form whose coefficients are weak * limits of the coefficients of the sequence $(Q_2^h)_{h>0}$.
We define the quadratic form $Q_0:\R^{2 \times 2} \to \R$ by
\begin{equation}\label{defqav}
Q_{0}(A)= \min_{B \in \mathbb{S}_2} \int_{I}  Q_2(x_3,A,B) dx_3.
\end{equation}

From (Q1) and (\ref{defQ2}) we easily establish that $Q_{0}$ is also a quadratic form satisfying the following
\begin{equation*}
\frac{\eta_1}{12}|\sym A |^2\leq Q_0 (A)=Q_0(\sym A) \leq \frac{\eta_2}{12}|\sym A|^2  \quad  \mbox{for all} \quad  A \in \mathbb{S}_{2}.
\end{equation*}
We  define the set of Sobolev  isommetries $$W^{2,2}_{iso}(S)=\left\{y \in W^{2,2}(S;\mathbb{R}^3): \partial_{\alpha}{y} \cdot \partial_{\beta}{y}=\delta_{\alpha\beta} \mbox{ a.e. in } S \right\},$$
and the limit functional $I_0:W^{2,2}_{iso}(S) \to \R$
$$I_0(y)= \int_S Q_{0}(II(x')) dx',$$
where $II(x')= (\nabla y)^T \nabla n$ is a second fundamental form and $n$ is a unit outer normal.

\subsection{Theorem on Geometric Rigidity and convergence of sequences with finite bending energy}

We will need the following lemmas.

\begin{lemma}[\cite{FJM-06}, Theorem 6]\label{lemaconv}
For any $y \in H^1(\Omega;\mathbb{R}^3)$ there are  maps $R:S \to SO(3)$ and $\tilde{R} \in H^1(S;\mathbb{R}^{3 \times 3})$ such that
$$\Vert \nabla_h y -R\Vert_{L^2(\Omega)}^2 + \Vert R - \tilde{R} \Vert_{L^2(\Omega)}^2 + h^2 \Vert \nabla \tilde{R} \Vert^2_{L^2(S)} \leq C\Vert \dist{(\nabla_h y,SO(3)) }\Vert^2_{L^2(\Omega)}.$$

\end{lemma}

\begin{lemma}[\cite{FJM-02}, Theorem 4.1]\label{lemastrong}
Suppose that a sequence $(y^h)_{h>0} \subset  H^1(\Omega;\mathbb{R}^3)$ is such that (\ref{bendingregime}) holds. Then there exists $y \in W^{2,2}_{iso}(S)$
such that on a subsequence, not relabeled
\begin{equation}
y^h - \frac{1}{\vert \Omega \vert } \int_{\Omega} y^h dx \to y \mbox{ strongly in } L^2(\Omega;\mathbb{R}^{3})  \label{jakayh}.
\end{equation}
We also have
\begin{equation}
\nabla_h y^h  \to (\nabla'y,n) \mbox{ strongly in } L^2(\Omega;\mathbb{R}^{3 \times 3}). \label{jakagradijenti}
\end{equation}

\end{lemma}

\subsection{Lower bound}
Before we prove the lower bound we will need the following claim.
\begin{lemma}(\cite{BoceaFon02,BraidesZeppieri07})\label{drugavazna}
Let $(w^h)_{h>0}$ be a sequence
bounded in $W^{1,p}(\Omega;\R^m)$, where $p>1$, and let us additionally assume that the sequence
$(\|\nabla_h w^h\|_{L^p})_{h>0}$ is bounded.
Then for every sequence $(w^{h_n})_{n \in \N}$ there exists a  subsequence $(w^{h_{n(k)}})_{k \in \N}$  such that  for every $k \in \N$ there exists $z_{k} \in W^{1,p}(\Omega;\R^m)$  which satisfies
\begin{enumerate}[(i)]
\item
$ \lim_{k \to \infty} |\Omega  \cap \{ z_{k} \neq w^{h_{n(k)}} \textrm{ or } \nabla z_{k} \neq  \nabla w^{h_{n(k)}} \}|=0. $
\item  $(|\nabla_{h_{n(k)}} z_{k}|^p)_{k \in \N}$ is equi-integrable.
\end{enumerate}
\end{lemma}
The following theorem states the lower bound claim.
\begin{theorem}\label{theorem:low}
Let Assumption \ref{ass:main} be valid.
For any sequence $(y^{h})_{h>0} \subset H^1(\Omega;\mathbb{R}^3)$  satisfying (\ref{bendingregime}) and (\ref{jakayh}),  we have
\begin{equation}
\liminf_{h \to 0} \frac{1}{h^2}I^{h}(y^{h}) \geq I_0(y).
\end{equation}
\end{theorem}

\begin{proof}
We take the subsequence $(h_k)_{k \in \N}$ such that liminf is attained.
\begin{enumerate}[1.]
\item  We define the sequence $(z^{h_k})_{k \in \N}$ by the following decomposition of $y^{h_k}$:
\begin{equation}\label{seq}
y^{h_k} = \bar{y}^{h_k} + h_k x_3 \tilde{R}^{h_k} e_3 + h_k z^{h_k}.
\end{equation}
Note  that $\bar{z}^{h_k} = 0$. By using Lemma \ref{lemaconv}  and the relation
$$\nabla_{h_k} z^{h_k} =\frac{\nabla_{h_k} y^{h_k}-R^{h_k}}{h_k}- \left(\frac{\nabla'\bar{y}^{h_k} -(R^{h_k})'}{h_k} + x_3 \nabla'\tilde{R}^{h_k} e_3\vert \frac{1}{h_k} (\tilde{R}^{h_k} e_3 - R^{h_k} e_3)  \right), $$
we obtain that $\Vert \nabla_{h_k} z^{h_k}\Vert_{L^2}$ is bounded. Also, using the Ponicar\'{e}'s inequality over the cross-sections, we derive that $z^{h_k} \to 0$ strongly in $L^2(\Omega)$.
\item Following \cite{FJM-02} we define the approximate strain by
\begin{align*}
G^{h_k}&=\frac{(R^{h_k})^t\nabla_{h_k} y^{h_k}-\id}{h_k}.
\end{align*}
We rewrite $G^{h_k}$ in the following form
$$
G^{h_k}=B^{h_k}(x') + x_3 A^{h_k}(x') + (R^{h_k})^t \nabla_{h_k} z^{h_k},
$$
where $A^{h_k}$ and $B^{h_k}$ are defined by
\begin{eqnarray}
\label{defah} A^{h_k} &=& (R^{h_k})^t(\partial_1{\tilde{R}^{h_k}} e_3 \vert \partial_2{\tilde{R}^{h_k}} e_3 \vert 0), \\
\label{defbh} B^{h_k} &=& \frac{1}{h_k}\left((R^{h_k})^t(\nabla'\bar{y}^{h_k} \vert \tilde{R}^{h_k} e_3) -\id \right).
\end{eqnarray}

Since, by Lemma \ref{lemaconv} and Lemma  \ref{lemastrong}, $R^{h_k} \to (\nabla' y, n)=:R$ strongly  and $\nabla\tilde{R}^{h_k} e_3 \rightharpoonup \nabla' n$ weakly in $L^2(S)$ we obtain that
\begin{align*}
A^{h_k} \rightharpoonup
\left(
\begin{array}{cc}
II(x') & \begin{array}{c}
0 \\
0
\end{array} \\
\nabla'n \cdot n & 0
\end{array}
\right) = : A \mbox{ weakly in } L^2(S;\mathbb{R}^{3 \times 3}),
\end{align*}
where $II_{\alpha \beta}(x')= \partial_{\alpha}{y} \cdot  \partial_\beta{n}$ is the second fundamental form. Also by the inequalities in Lemma \ref{lemaconv} we obtain that on a subsequence (not relabeled)
there exists $B' \in L^2(S;\mathbb{R}^{3 \times 3})$
\begin{align*}
B^{h_k} \rightharpoonup
B' \mbox{ weakly in } L^2.
\end{align*}

\item Denote by  $\chi_{k}:S \to \{0,1\}$  the characteristic function:
\begin{equation*}
\chi_k(x)= \left\{
\begin{array}{cc}
1,  &  x \in S_{h_k}^{'} \cap \{ \vert G^{h_k} \vert \leq h_k^{-\frac{1}{2}} \}, \\
0, & \mbox{ else. }
\end{array} \right.
\end{equation*}
Notice that $\chi_k \to 1$ boundedly in measure.
Taking into account the frame-indifference of $W^h$ and (\ref{assumQh}) we conclude
\begin{eqnarray*}
\liminf_{k \to \infty} \tfrac{1}{h_{k} ^2} \int_{\Omega} W^{h_{k}}(x,\nabla_{h_{k}} y^{h_{k}}) &=& \liminf_{k \to \infty} \tfrac{1}{h_{k} ^2} \int_{\Omega} W^{h_{k}}(x,\id+h_kG^{h_k}) \\
&\geq &  \liminf_{k \to \infty} \int_{\Omega} \chi_k Q^{h_{k}}(x_3, G^{h_k}).
\end{eqnarray*}

\item By Lemma \ref{lemaconv} we know that $R^{h_k} \to R$ converge in $L^2(\Omega)$ strongly. Using Lemma \ref{drugavazna} we conclude that there is a subsequence   $(\hat{z}^{h_k})_{k \in \N} $, not relabeled, such that
\begin{enumerate}[(i)]
\item $ \lim_{k \to \infty} |\Omega  \cap \{ \hat{z}^{h_k} \neq z^{h_{k}} \textrm{ or } \nabla \hat{z}^{h_k} \neq  \nabla z^{h_{k}} \}|=0$;
\item  $(|\nabla_{h_{k}}(\hat{z}^{h_k})\vert^2 )_{k \in \N}$ is equi-integrable.
\end{enumerate}
Since $R^h$ is bounded in $L^{\infty}$ norm   and $(|\nabla_{h_k} \hat{z}^{h_k}|^2)_{k \in \N}$ is an equi-integrable sequence we conclude, by using the Egoroff's theorem,  that:
\begin{equation}\label{ocjenaRh}
(R^{h_k})^t \nabla_{h_k} \hat{z}^{h_k} -  R^t \nabla_{h_k}  \hat{z}^{h_k}\to 0 \mbox{
 strongly in } L^2(\Omega).
 \end{equation}
Next, we can construct the sequence $(M^{h_k})_{k \in \N} \subset C^\infty_c(S,\mathbb{R}^{3 \times 3})$ such that
\begin{enumerate}[(i)]
\item $M^{h_k} \to R \mbox{ in } L^2(S)$;
\item $\Vert M^{h_k}\Vert_{L^{\infty}(S)} \leq C$, for some $C>0$ independent of $h_k$;
\item $\lim_{h \to 0} \Vert \nabla' M^{h_k} \Vert_{L^\infty} \Vert \hat{z}^{h_k}\Vert_{L^2} \to 0$.
\end{enumerate}
We conclude  that
\begin{equation}\label{ocjenaMh}
(M^{h_k})^t \nabla_{h_k} \hat{z}^{h_k} -  R^t \nabla_{h_k} \hat{z}^{h_k}\to 0 \mbox{
 strongly in }L^2(\Omega)
 \end{equation}
 and define a sequence $(\tilde{z}^{h_k})_{k \in \N}$ by $\tilde{z}^{h_k}= (M^{h_k})^t \hat{z}^{h_k}$. From the relation
\begin{equation}\label{relacijamMh}
(M^{h_k})^T \nabla_{h_k} \hat{z}^{h_k} = \nabla_{h_k}\left((M^{h_k})^t \hat{z}^{h_k} \right) - (\nabla' M^{h_k})^t \hat{z}^{h_k},
\end{equation}
we obtain that
\begin{equation}\label{relacijamaroje11}
\| \nabla_{h_k} \tilde{z}^{h_k} -(R^{h_k})^t \nabla_{h_k} \hat{z}^{h_k} \|_{L^2} \to 0,
\end{equation}
and thus the sequence $(| \nabla \tilde{z}^{h_k}|^2)_{k \in \N} $ is also equi-integrable. Combining  (Q2) with the relations  (\ref{ocjenaRh}), (\ref{ocjenaMh}),
(\ref{relacijamMh}) and equi-integrability  we obtain
$$\liminf_{k \to \infty} \int_{\Omega} \chi_k Q^{h_{k}}(x_3, G^{h_k})
= \liminf_{k \to \infty} \int_{\Omega}  Q^{h_{k}} \left(x_3, \chi_k(B^{h_k} + x_3 A^{h_k}) + \nabla_{h_{k}} \tilde{z}^{h_{k}} \right).$$
Without loss of generality we may assume that $\tilde{z}^{h_k}$ are smooth functions.
\item Next we approximate $S$ by the union of squares $S^\eps= \bigcup_{i} D_i^\eps$ parallel to the axis whose edge length is equal to $\eps$ and $\vert S\setminus S^\eps \vert \to 0$ as $\eps \to 0$.   Using the Jensen's inequality and the convexity of $Q^{h_k}(x_3,\cdot)$ we obtain for every $\eps>0$
\begin{eqnarray*}
 &&\liminf_{k \to \infty} \int_{\Omega}  Q^{h_{k}} \left(x_3, \chi_{k}(B^{h_k} + x_3 A^{h_k}) + \nabla_{h_{k}} \tilde{z}^{h_{k}} \right)dx \nonumber  \\
 &  \geq  &  \liminf_{k \to \infty}\int_I  \sum_{i}  \vert\D_i^{\eps}\vert  Q^{h_k}\left(x_3,\frac{1}{\vert D_{i}^{\eps} \vert}\int_{D_i^{\eps}} \left(\chi_{k}( B^{h_k} +   x_3 A^{h_k})+  \nabla_{h_k}{\tilde{z}^{h_k}} \right)dx' \right) dx_3.
\end{eqnarray*}
Since $B^{h_k} \rightharpoonup  B'$ and $A^{h_k} \rightharpoonup   A$ converge weakly in  $L^2(S;\R^{3 \times 3})$   we derive that as $k \to \infty$
$$\int_{D_i^{\eps}} \chi_k \left(  B^{h_k} + x_3 A^{h_k} \right) \to  \int_{D_i^{\eps}}( B' + x_3 A) \mbox{ on every } D_i^{\eps}. $$

From the compactness theorem for trace for such functions   we obtain that for every $\eps>0$
$$\left(\int_{D_{i}^{\eps} \times I} \partial_1{\tilde{z}^{h_k}},\int_{D_{i}^{\eps}} \partial_2{\tilde{z}^{h_k}}  \right) = \left(\int_{\partial{D_{i}^{\eps}}} \tilde{z}^{h_k} \cdot n_1 ,\int_{\partial{D_{i}^{\eps}}} \tilde{z}^{h_k} \cdot n_2  \right) \to (0,0) \mbox{  in  }
L^2(I),
$$
for each $D_{i}^{\eps}$.  We denote by
$\psi^{h_k}(x_3)= \frac{1}{h_k} \partial_3 \int_{D_{i}^{\eps}}\tilde{z}^{h_k}dx'$ and by $B$ the symmetric part of the submatrix of $B'$ obtained by removing the third row and column.  Using the definition  of $Q^{h_k}_2$ we conclude that for every $\eps>0$
\begin{eqnarray*}
  &  &   \liminf_{k \to \infty}\int_I  \sum_{i}  \vert D_i^{\eps} \vert Q^{h_k}\left(x_3,\frac{1}{\vert D_i^{\eps} \vert}\int_{D_{i}^{\eps}}\left(  \chi_{k} (B^{h_k} +   x_3 A^{h_k}) +  \nabla_{h_k}{\tilde{z}^{h_k}} \right) dx'\right) dx_3  \\
 &  = &  \liminf_{k \to \infty}\int_I  \sum_{i} \vert D_i^{\eps} \vert Q^{h_k}\left(x_3,\frac{1}{\vert D_i^{\eps} \vert}\int_{D_i^{\eps}} \left( B' +   x_3 A + (0,0,\psi^{h_k}(x_3)) \right) dx'\right) dx_3 \\
 &  \geq & \liminf_{k \to \infty}\int_{I}  \sum_{i} \vert D_i^{\eps} \vert Q_{2}^{h_k}\left(x_3,\frac{1}{\vert D_i^{\eps}\vert}\int_{D_{i}^{\eps}} (B + x_3 II) dx' \right)  dx_3\\
  &  \geq &   \sum_{i} \vert D_i^{\eps} \vert \liminf_{k \to \infty}\int_{I}   Q_{2}^{h_k}\left(x_3,\frac{1}{\vert D_i^{\eps}\vert}\int_{D_{i}^{\eps}} (B + x_3 II) dx' \right)  dx_3 \\
  &\geq&   \sum_{i} \vert D_i^{\eps} \vert    Q_{0}\left(\frac{1}{\vert D_i^{\eps} \vert} \int_{D_{i}^{\eps}} II(x') dx' \right).
\end{eqnarray*}
Finally noting that $\chi_{S^\eps} \left(\sum_i \frac{1}{\vert D_i^\eps \vert}\int_{D_i^\eps} f \right) \to f $ strongly in $L^2(S)$ for each $f \in L^2(S)$  and taking the limit as $\eps \to 0$ we finish the proof.

\end{enumerate}
\end{proof}
\subsection{Upper bound}
Denote by $\mathcal{A}(S)$ the set of all $y \in W^{2,2}(S,\R^3) \cap C^{\infty}(\bar{S},\R^3)$ with the property: for each $B \in C^{\infty}(\bar{S},\mathbb{S}_2)$ satisfying $B=0$ in a neighborhood of $\{x \in S:II(x)=0 \}$ there are $\alpha \in C^{\infty}(\bar{S})$ and $g \in C^{\infty}(\bar{S},\mathbb{R}^2)$ such that
\begin{equation} \label{eq:zzadnje}
B=\sym{\nabla'g} + \alpha II.
\end{equation}
We will use the following lemma, already used in \cite{Horneuvel12} and \cite{Velcic-13}.
It was  proved in \cite{Schmidt-07} for convex sets $S$. By using the results from \cite{Hornung-arma2} we can extend the claim.
\begin{lemma}\label{lemma:density}
Assume that $S$ is a bounded Lipschitz domain with boundary which is piecewise $C^1$.
The set $\mathcal{A}(S)$ is dense in $W^{2,2}_{iso}(S)$ with respect to the strong  $W^{2,2}(S)$ topology.
\end{lemma}

We now construct the recovery sequence.
\begin{theorem}
Let Assumption \ref{ass:main} be valid.
For any  $y \in W^{2,2}_{iso}(S)$ there is a sequence $(y^{h})_{h>0} \subset H^1(\Omega;\mathbb{R}^3)$  such that $I^{h}(y^{h}) \leq C h^2 $, $y^{h} \to y$, strongly in $L^2(\Omega)$ and
\begin{equation}
\lim_{h \to 0} \frac{1}{h^2}I^{h}(y^{h})= I_0(y).
\end{equation}
\end{theorem}
\begin{proof}
\begin{enumerate}[1.]
\item \emph{Definition of the recovery sequence.} Since $I_0$ is continuous in $W^{2,2}(S)$ topology,  by Lemma \ref{lemma:density} we may assume  that $y \in \mathcal{A}(S)$. We proceed as in \cite{Schmidt-07}. First we take  arbitrary $B \in C^{\infty}(\bar{S},\mathbb{S}_2)$ such that $B=0$ in a neighborhood of $\{x \in S:II(x)=0 \}$.  We denote by $R(x') = (\nabla'u(x'),n(x'))$ and define the recovery sequence by
\begin{equation*}
y^{h,\eps}_{\delta}(x) = y(x') + h\left(x_3 + \alpha(x')) n(x') + (g(x')\cdot \nabla')y(x')\right)  + h^2 D^{h}(x',x_3),
\end{equation*}
where $\alpha,g$ are determined by the relation (\ref{eq:zzadnje}) and the correctors are of the form  $D^{h}(x',x_3)=\int_{0}^{x_3} R(x') d^{h}(x',t) dt$. The precise construction of $d^{h}$ is given in the next step.

The approximate strain for the sequence $y^{h}$ then equals
\begin{align*}
G^{h} &= \frac{1}{h}\left(R^t \nabla_{h}y^h - \id \right) = R^t \left(\nabla'\left((x_3+\alpha)n + (g \cdot\nabla')y \right) , R d^{h} \right) +hR^t(\nabla'{D^{h}},0).
\end{align*}
Taking only the symmetric part we obtain that
$$\sym G^{h} = \left(\begin{array}{cc}
x_3II(x') + B & \begin{array}{c} \frac{1}{2}(d_1^{h} -b_1) \\ \frac{1}{2}(d_2^{h} -b_2) \end{array} \\
 \begin{array}{cc} \frac{1}{2}(d_1^{h} -b_1)
& \frac{1}{2}(d_2^{h} -b_2)
 \end{array} & d_3^{h}
\end{array}
\right) +h \sym(R^t(\nabla'D^{h},0))
$$
where $b=-\left(\begin{array}{c}
\partial_1{\alpha} \\
\partial_{2}{\alpha}
\end{array} \right) + II g$.

\item \emph{Construction of the correctors.}  As in the proof of theorem \ref{theorem:low} for every $\eps>0$ we approximate $S$ by the set  $S^\eps=\bigcup_i D_i^{\eps}$ where $D_i^{\eps}$ are squares parallel to the axis whose edge length is equal to $\eps$ and $\lim_{\eps \to 0}\vert S \setminus S^\eps \vert = 0$.

On each square $D_i^{\eps}$ we define the average values of $B$, $II$ and $b$ by:
$$
\tilde{B}^{\eps} =\frac{1}{\vert D_i^{\eps}\vert} \int_{D_i^{\eps}} B(x') \ dx',\
\tilde{II}^{\eps}=\frac{1}{\vert D_i^{\eps}\vert} \int_{D_i^{\eps}} II(x') \ dx', \
\tilde{b}^{\eps} =\frac{1}{\vert D_i^{\eps}\vert} \int_{D_i^{\eps}} b(x')\ dx'.
$$
We extend the definition of these functions to the whole set $S$ by defining them to be zero on $S \setminus S^{\eps}$. It is easy to see that
\begin{equation}\label{igorrr1}
\tilde{B}^{\eps} \to B, \
\tilde{II}^{\eps} \to II,\
\tilde{b}^{\eps} \to b,
\end{equation}
strongly in $L^{2}(S)$  and that they are bounded in $L^{\infty}(S)$.

Since by definition $\tilde{II}^{\eps}$,$\tilde{B}^{\eps}$ and $\tilde{b}^{\eps}$ are constants on every $D_i^{\eps}$,  for each $h$ we define $d^{h,\eps}_{min}:S^{\eps} \to \mathbb{R}^3$  by (\ref{defQ2}).  Using (Q1) with the zero sequence on the right side we conclude that for each $h$ we have
\begin{equation}
\eta_1 \vert \tilde{B}^{\eps} + x_3 \tilde{II}^{\eps} + d_{min}^{h,\eps} \otimes e_3 + e_3 \otimes d_{min}^{h,\eps}\vert^2 \leq  \eta_2 \vert \tilde{B}^{\eps} + x_3 \tilde{II}^{\eps} \vert^2.
\end{equation}
Since both $(\tilde{B}^{\eps})_{\eps>0}$ and $(\tilde{II}^{\eps})_{\eps>0}$ are uniformly bounded we conclude that $d^{h,\eps}_{min}$ is uniformly bounded with respect to $\eps$ and $h$. We extend the definition of $d^{h,\eps}_{min}$ to the whole of $\Omega$ by assuming that $d^{h,\eps}_{min}=0$ on $S \setminus S^{\eps}$.

Moreover, we approximate $d^{h,\eps}_{min}$ by smooth functions $d^{h,\eps}_{\delta}$ depending on some positive parameter $\delta$. For each $h, \eps$, a.e. $x_3\in I$ and $\delta \leq \frac{\eps}{2}$ we construct smooth functions $d_{\delta}^{h,\eps }$ on $S^{\eps}$ such that:
\begin{enumerate}
\item[(i)] $d_{\delta}^{h,\eps} = d^{h,\eps}_{min}$ on a set $(x^{i,\eps}_{1}+\delta,x^{i,\eps}_{1}+\eps-\delta)\times (x^{i,\eps}_{2}+\delta,x^{i,\eps}_{2}+\eps-\delta)$ where $(x^{i,\eps}_{1},x^{i,\eps}_{2})$ are coordinates of the lower left corner of $D_i^{\eps}$;
\item[(ii)]  for each $h,\eps$ we have that $d^{h,\eps}_{\delta} \to d^{h,\eps}_{min}$ strongly in $L^2(\Omega)$ as $\delta \to 0$;
\item[(iii)] there is a positive constant $C>0$ independent of $\eps$ and $h$ such that $\Vert \nabla' d^{h,\eps}_{\delta}  \Vert_{L^\infty(\Omega)} \leq \frac{C}{\delta}$;
\item[(iv)] $d^{h,\eps}_{\delta}$ is uniformly bounded with respect to $h,\delta$ and $\eps$.
\end{enumerate}
We define the recovery sequence by
\begin{equation*}
y^{h,\eps}_{\delta}(x) = y(x') + h\left((x_3 + \alpha(x')) n(x') + (g(x')\cdot \nabla')y(x')\right)  + h^2 D^{h,\eps}_{\delta}(x',x_3),
\end{equation*}
where the correctors are of the form  $D^{h,\eps}_{\delta}(x',x_3)=\int_{0}^{x_3}
R(x') d^{h,\eps}_{\delta}(x',t) dt$. We denote by $G^{h,\eps}_{\delta}$ the approximate strain generated by $y^{h,\eps}_{\delta}$. The condition (iii) ensures the boundedness of $G^{h,\eps}_{\delta}$  as soon as we have the boundedness of $\frac{h}{\delta}$.


\item \emph{Passing  to the limit}.
Taking into account the frame indifference property  of $W^{h}$ and (\ref{assumQh}) we obtain that:
\begin{eqnarray*}
\lim_{h \to 0}  \frac{1}{h^2} \int_{\Omega}W^h(x_3,\nabla_{h} y^{h,\eps}_{\delta}) &=& \lim_{h \to 0}  \frac{1}{h^2} \int_{\Omega} W^{h}(x_3,\id+h G^{h,\eps}_{\delta} ) \\
&=& \lim_{h \to 0} \int_{\Omega} Q^{h}(x_3,G^{h,\eps}_{\delta}).
\end{eqnarray*}
Using this the boundedness of $G^{h,\eps}_{\delta}$ with respect to $\eps$ we obtain
$$\lim_{\eps \to 0} \lim_{\delta \to 0} \lim_{h\to 0} \int_{ (S \setminus S^{\eps}) \times I } Q^{h}(x_3,G^{h,\eps}_{\delta}) = 0.
$$
We  prove the convergence on sets $S^{\eps} \times I$. We define $\tilde{G}_{\delta}^{h,\eps}$ from $G^{h,\eps}_{\delta}$ by replacing $B,II$ and $b$ by $\tilde{B}^{\eps},\tilde{II}^{\eps}$ and $\tilde{b}^{\eps}$. Note that using  (Q2) we obtain that there is a constant $C>0$ independent of $h,\eps,\delta$  such that:
\begin{align*}
 \left\vert   \int_{S^{\eps} \times I} Q^{h}(x_3,G^{h,\eps}_{\delta}) -\int_{S^{\eps} \times I} Q^{h}(x_3,\tilde{G}^{h,\eps}_{\delta}) \right\vert \leq C\left(1+\frac{h}{\delta}\right) \Vert G^{h,\eps}_{\delta}- \tilde{G}^{h,\eps}_{\delta} \Vert_{L^2(S^{\eps}\times I)}.
\end{align*}
From the definition of the sequence $d^{h,\eps}_{\delta}$  we obtain that $Q^{h}(x_3,\tilde{G}^{h,\eps}_{\delta}) = Q^{h}_2(x_3,\tilde{II},\tilde{B})$ on every $D_i^{\eps}$ except on the band of the thickness $\delta$. For every $\delta$ we denote this set by $$S^{\eps}_{\delta} = \bigcup_{i}D_i \setminus (x^{i,\eps}_1 + \delta ,x^{i,\eps}_1+\eps- \delta ) \times (x^{i,\eps}_2 + \delta ,x^{i,\eps}_2+\eps- \delta ).$$ Thus, we obtain that for some $C>0$ independent of $h,\delta,\eps$
\begin{align*}
\left\vert \int_{S^\eps \times I}  Q^{h}(x_3,\tilde{G}^{h,\eps}_{\delta}) -Q^{h}_2(x_3,\tilde{II}^{\eps},\tilde{B}^{\eps}) dx \right\vert &=  \left\vert \int_{S^{\eps}_{\delta} \times I}  Q^{h}(x_3,\tilde{G}^{h,\eps}_{\delta}) -Q^{h}_2(x_3,\tilde{II}^{\eps},\tilde{B}^{\eps}) dx \right\vert \\
&\leq C  \int_{S^{\eps}_{\delta} \times I}   (\vert \tilde{G}^{h,\eps}_{\delta}\vert^2 + \vert x_3\tilde{II}^{\eps}+\tilde{B}^{\eps}\vert^2) dx \\
&\leq C \vert S^{\eps}_{\delta} \vert  \left( 1+\Vert d^{h,\eps}_{\delta}\Vert^2_{L^{\infty}} + \frac{h^2}{\delta^2}+ 2 \Vert x_3\tilde{II}^{\eps}+\tilde{B}^{\eps}\Vert^2_{L^{\infty}} \right).
\end{align*}
By the definition of $Q_2$ and (\ref{igorrr1}) it is easy to see that
$$\lim_{\eps \to 0} \lim_{h \to 0} \int_{S^{\eps} \times I}Q_2^{h}(x_3,\tilde{II}^{\eps},\tilde{B}^{\eps}) dx =  \int_{S \times I }Q_{2}(II,B) dx.
$$
We define the function
\begin{equation*}
g(\eps,\delta,h):= \left\vert \frac{1}{h^2} I^{h}(y^{h,\eps}_{\delta}) - \int_{\Omega} Q_{2}\left(II,B \right) dx \right\vert +
\Vert y^{h,\eps}_{\delta} -y\Vert_{L^2(\Omega)}+\frac{h}{\delta}.
\end{equation*}
Using the above estimates  we obtain that  $\lim_{\eps \to 0}\lim_{\delta \to 0} \lim_{h \to 0} g(\eps,\delta,h)=0$. Hence, by Attouch lemma, there are functions $\eps(h),\delta(h)$, such that $\eps(h),\delta(h) \to 0$ as $h \to 0$ and
$$\lim_{h \to 0}  g(\eps(h),\delta(h),h)=0.$$

We can choose $(B_k)_{k \in \N}\subset L^2(S;\R^{2 \times 2})$ smooth such that for every $k \in \N$ $B_k=0$ in a neighbourhood of the set $\{II=0\}$ and $B_k \to B_{min}$ in $L^2$,
where
$$ B_{min}(x')=\argmin_{B \in \mathbb{S}_2} \int_{I}  Q_2(x_3,II(x')) dx_3.$$
By the diagonalization  we obtain the result.

\end{enumerate}

\end{proof}
\begin{remark} \label{napomena1}
In \cite{Schmidt-07} it was proved that if we start from the layered materials (i.e., the stored energy density $W^{F}$, and thus its second derivative  at the identity $Q^{F}$, depends on $x_3$, but is fixed) then the energy density $\bar{Q}_2^{F}$ of $\Gamma$-limit has the form
$$\bar{Q}_2^{F}(A)= \min_{B\in S_2} \int_{I} Q_2 ^F  (x_3,B+x_3A),$$
where $Q_2^{F}$ is obtained from $Q^{F}$ by minimizing with respect to the third row and column as in the relation (\ref{defQ2}).
We have shown that in this more general case the limit energy density is determined by the quadratic form $Q_2$ which can be obtained as weak limit of quadratic forms $Q_2^h$ as $h$ goes to zero. Thus the effective behavior depends only on weak limit, i.e., averages of the oscillations in the direction of thickness. Because of this we can say that in this case there are no homogenization effects.
\end{remark}

\section{Periodic homogenization}
The result of the previous section is given in terms of general homogenization with respect to the   oscillations in the direction of thickness.  In this section we will assume that we have periodic oscillations in the direction of thickness together with the periodic oscillations in the in-plane directions. The reason why we study only periodic in-plane oscillations is the fact that we are not able to obtain the effective behavior for the case of general non periodic in-plane oscillations (see \cite{Vel13}). We will give only the sketches of the proofs since the basic ideas follow the ones in \cite{Horneuvel12,Vel13}.

We assume that the periodicity is given on two scales $\eps_1(h)$ and $\eps_2(h)$  where $\lim_{h \to 0}\eps_1(h) = \lim_{h \to 0}\eps_2(h)=0$ and use the notation  $Y^d=[0,1)^d$ where $d\in\{1,2,3\}$ for periodic cells. We denote by $\mathcal{Y}^d$ the sets $Y^d$ endowed with the torus topology.

We  give the assumptions on energy densities with periodically oscillating material.
 We assume that $W:\R^3\times\R^{3\times 3}\to [0,\infty]$
  is Borel measurable and that $W(\cdot, F)$ is $[0,1)^3$-periodic for all $F\in\R^{3\times 3}$.
  Furthermore, we assume that for a.e. $y \in [0,1)^3$ we have that $W(y,\cdot) \in \calW(\eta_1,\eta_2, \rho)$.  We also assume that there exists $Q:\R^3\times\R^{3\times 3}\to [0,\infty)$ such that, for a.e. $y \in [0,1)^3$ $Q(y,\cdot)$, is a quadratic form and the following is valid
 \begin{equation}\label{assumQh1}
    \ \forall \boldsymbol G\in\R^{3\times 3} \quad \esssup_{x_3\in I} |W(x_3,\id+\boldsymbol G)-Q(x_3,\boldsymbol G)|\leq r(|\boldsymbol G|) |\boldsymbol G|^2,
  \end{equation}
where $r:\R^+\to\R^+\cup\{+\infty\}$ is a monotone function  that satisfies $\lim_{\delta \to 0} r(\delta)= 0$. From the properties (W1)-(W4) we conclude that for a.e. $y \in [0,1)^3$ we have
\begin{equation} \label{svojstvoq}
 \eta_1 |\sym A|^2 \leq Q(y,A)=Q(y,\sym A) \leq \eta_2 |\sym A|^2, \forall A \in \R^{3 \times 3}.
 \end{equation}
\begin{definition}[Two-scale convergence with respect to two scales $\eps_1$ and $\eps_2$]
We say that a bounded sequence  $(u^h)_{h>0}$ in $L^2(\Omega)$ two-scale converges to $u \in L^2(\Omega \times Y^3)$
and we write $u^h(x) \stackrel{2}{\rightharpoonup} u(x,y) $  if
$$\lim_{h\to 0} \int_\Omega u^h(x) \psi\left(x,\frac{x'}{\eps_1(h)},\frac{x_3}{\eps_2(h)}\right) dx =
\iint_{\Omega \times \mathcal{Y}^3} u(x,y) \psi(x, y) dy dx,
$$
for all $\psi \in C^{\infty}_0(\Omega,\mathcal{C}(\mathcal{Y}^3))$. If,  in addition,
$\Vert u^h\Vert_{L^2(\Omega)}  \to  \Vert u \Vert_{L^2(\Omega \times \mathcal{Y}^3)}$ we say
that $u^h$ strongly two-scale converges to $u$ and write $u^h \stackrel{2}{\to} u$. For vector-valued
functions, two-scale convergence is defined componentwise.
\end{definition}
It is easy to see that all standard claims for two-scale convergence (such as compactness, lower semicontinuity of convex integrands)  are also valid in this case (for two-scale convergence see \cite{Allaire-92,Visintin-06,Visintin-07}).
Here we study two regimes in the periodic context.
\begin{enumerate}[1.]
\item  $\eps_1(h)=h^{\alpha+1}$ and $\eps_2(h) = h^{\alpha}$ for some $\alpha >  0$.
\item $\lim_{h \to 0}\frac{h}{\eps_1(h)}=1$ and $\lim_{h \to 0} \eps_2(h)=0$.
\end{enumerate}

\subsection{Heuristics}
We assume that scaling is given by $\eps_1(h)=h^{\alpha+1}$ and $\eps_2(h) = h^{\alpha}$ for $\alpha >  0$ and assume the following form for $u^h$
\begin{align*}
u^h\left(x,\frac{x'}{h^{\alpha+1}},\frac{x_3}{h^{\alpha}}\right) = u^0(x') + h u^1\left(x,\frac{x'}{h^{\alpha+1}},\frac{x_3}{h^{\alpha}}\right)
+  h^{\alpha+1} u^2\left(x,\frac{x'}{h^{\alpha+1}},\frac{x_3}{h^{\alpha}}\right).
\end{align*}
We derive that
\begin{align*}
\nabla_h u^h = (\nabla'u^0\vert\partial_{x_3}{u^1}) +  \frac{1}{h^{\alpha}}\nabla_{y}u^1 + \nabla_y u^2.
\end{align*}
From the assumptions on the relations between $h$, $\eps_1(h)$ and $\eps_2(h)$ and the boundedness of $\nabla_h u^h$ we derive that $\nabla_y u^1=0$. Thus, $u^1(x,y)=u^1(x)$ and  $\nabla_h u \to (\nabla'u^0\vert\partial_{x_3}{u^1}) + \nabla_y u^2$.

In the second case we assume that $u^h$ equals
\begin{align*}
u^h\left(x,\frac{x'}{\eps_1(h)},\frac{x_3}{\eps_2(h)}\right) = u^0(x') + h u^1\left(x,\frac{x'}{\eps_1(h)},\frac{x_3}{\eps_2(h)}\right)
+  h \eps_2(h) u^2\left(x,\frac{x'}{\eps_1(h)},\frac{x_3}{\eps_2(h)}\right).
\end{align*}

Then the scaled gradient equals
\begin{align*}
\nabla_h u^h = \left(\nabla'u^0 \vert \partial_{x_3}{u^1}\right) + \frac{h}{\eps_1(h)}\left(\nabla_y'u^1\vert 0 \right) + \frac{1}{\eps_2(h)}\left(0\vert 0 \vert \partial_{y_3} u^1\right) + \frac{ h \eps_2(h)}{\eps_1(h)}\left(\nabla_{y}' u^2 \vert 0 \right) +\left(0 \vert 0 \vert \partial_{y_3} u^2 \right) .
\end{align*}
Since $\nabla_h u^h$ is bounded we derive that $u^1(x,y)=u^1(x,y')$ and that the two-scale limit of scaled gradient equals
$$\left(\nabla'u^0 \vert  0 \right) +\left( \nabla_{y'}u^1\vert \partial_{x_3}{u^1} \right)  +\left(0 \vert \partial_{y_3} u^2 \right). $$

We give the proofs of these claims.

\begin{proposition}\label{prop:two-scale}
Let $(u^h)_{h>0}$ be a weakly convergent sequence in $H^1(\Omega,\mathbb{R}^3)$ with limit $u$ and assume that $(\Vert \nabla_h u^h\Vert_{L^2(\Omega)})_{h>0}$ is bounded.
\begin{enumerate}[1.]
\item  Suppose that $\eps_{1}(h)=h^{\alpha+1}$ and $\eps_2(h)=h^{\alpha}$ for some $\alpha>0$.
 Then $u$ is independent of $x_3$ and there are functions $u^1 \in H^1(\Omega;\mathbb{R}^3)$ and $u^2 \in L^2(\Omega;H^1( \mathcal{Y}^3;\mathbb{R}^3))$
such that
$$\nabla_h u^h \stackrel{2}{\rightharpoonup} (\nabla'u^0\vert\partial_{x_3}{u^1})
+ \nabla_y u^2,
$$
weakly two-scale in $L^2(\Omega\times Y^3;\mathbb{R}^3)$ on a subsequence.
\item Suppose that $\lim_{h \to 0}\frac{h}{\eps_1(h)}=1$ and $\lim_{h \to 0} \eps_2(h)=0$.
 Then $u$ is independent of $x_3$ and there are functions $u^1 \in L^2(\Omega;H^1(\mathcal{Y}^2;\mathbb{R}^3))$ and $d\in L^2(\Omega \times Y^3;\mathbb{R}^3)$ such that $\int_{Y} d(x,y',y_3) dy_3= 0$ for a.e. $(x,y')\in \Omega \times Y^2$
and
$$\nabla_h u^h \stackrel{2}{\rightharpoonup} \left(\nabla'u^0 \vert  0 \right) + (\nabla_{y'} u^1 \vert \partial_{x_3} u^1 )+  \left(0 \vert 0 \vert d \right),
$$
weakly two-scale in $L^2(\Omega\times Y^3;\mathbb{R}^3)$ on a subsequence.
\end{enumerate}
\end{proposition}


\begin{proof}
We follow the ideas in \cite{Neukamm-10}, Theorem 6.3.3.  and \cite{Horneuvel12}. We define  $\tilde{u}^h=u^h-\bar{u}^h$. Since $\tilde{u}^h$ has a zero mean value,  by the
Poincar\'{e}'s inequality, there is a  constant $C>0$ independent of $h$ such that
$$\int_I \vert \tilde{u}^h \vert^2 dx_3 \leq C \int_{I} \vert\partial_{x_3}{u^h} \vert^2 dx_3,$$
for a.e. $x' \in \omega$.  Integrating over $\omega$ we obtain that
\begin{equation}\label{ineq:pnc}
\int_\Omega \vert \tilde{u}^h \vert^2 dx_3 \leq C \int_{\Omega} \vert\partial_{x_3}{u^h} \vert^2 dx_3.
\end{equation}
Since $\nabla_h u^h$ is bounded we derive that $\partial_3{u^h} \to 0$ in $L^2(\Omega;\mathbb{R}^3)$. Thus, $\tilde{u}^h \to 0$ strongly  in $L^2(\Omega)$. Since  $\bar{u}^h $ weakly converges in $H^1(S;\R^3)$ we conclude that $u^h$ converges to some $u^0 \in H^1(S;\R^3)$.

From (\ref{ineq:pnc}) we obtain that
\begin{equation}
\left\Vert \frac{1}{h}\tilde{u}^h \right\Vert_{L^2(\Omega)}  \leq
C \left\Vert \frac{1}{h} \partial_3{u^h} \right\Vert_{L^2(\Omega)}.
\end{equation}
Since $(\nabla_h u^h)_{h>0}$ is bounded in $L^2$, there is a function $u^1 \in L^2(\Omega;\R^3)$ with $\partial_3 u^1 \in L^2(\Omega;\R^3)$ and $\bar{u}^1=0$ such that  $\frac{1}{h}(\tilde{u}^h) \rightharpoonup u^1$ and $\frac{1}{h} \partial_3{\tilde{u}^h} \rightharpoonup \partial_3{u^1}$ converge weakly in $L^2$ on a subsequence.

Without loss of generality, for subsequent analysis we will assume that $u^0=u^1=0$ (otherwise we subtract from $u^h$ the term $u^0+hu^{1,h}$, where
$u^{1,h}$ is smooth and $\|u^{1,h}-u^1\|_{L^2} \to 0$, $\|\partial_3 u^{1,h}-\partial_3u^1\|_{L^2} \to 0$, $h \|\partial_{\alpha} u^{1,h}\|_{L^2} \to 0$).

\begin{enumerate}[1.]
\item    Assuming that $\nabla_h u^h \to 0$ in $L^2(\Omega)$ it remains to prove that there is a function $u^2 \in L^2(\Omega;H^1(\mathcal{Y}^3;\R^3) )$ such that
$$\nabla_h u^h \stackrel{2}{\rightharpoonup}  \nabla_{y} u^2, $$
weakly two-scale in $L^2(\Omega\times \mathcal{Y}^3;\mathbb{R}^3)$ on a subsequence.

Since bounded sequences are weakly two-scale compact, there is a function $U \in L^2(\Omega\times Y^3;\R^{3 \times 3})$ such that $\nabla_h u^h \stackrel{2}{\rightharpoonup} U$, i.e.,
$$
\iint_{\Omega \times Y^3}  U(x,y) \cdot \Psi(x,y) dy dx =
\lim_{h \to 0} \int_{\Omega} \nabla_h u^h  \cdot \Psi\left(x,\frac{x'}{h^{\alpha+1}},\frac{x_3}{h^{\alpha}}\right)  dx.
$$
Moreover, we can conclude that $\int_{Y^3} U(x,y) dy=0$ for a.e. $x \in \Omega$.
By taking the test functions $\Psi \in C^{\infty}_0(\Omega,C^{\infty}(\mathcal{Y}^3;\R^{3 \times 3}))$ such that $\diver_{y} \Psi = 0$  and $\int_{Y^3} \Psi(x,y) dy=0$ for every $x \in \Omega$, after partial integration we obtain that
\begin{align*}
\int_{\Omega}  \nabla_h u^h \cdot \Psi dx &= \int_{\Omega} u^h \cdot \left(\partial_{x_1}{\Psi_1} + \partial_{x_2}{\Psi_2} +\frac{1}{h}\partial_{x_3}{\Psi_3} \right)dx \\
&+ \frac{1}{h^{\alpha +1}} \int_{\Omega} u^h \cdot \left(  \partial_{y_1}\Psi_1 +\partial_{y_2}\Psi_2  +  \partial_{y_3}\Psi_3 \right) dx.
\end{align*}
The second term is equal to zero. Taking the limit as $h \to 0$ and using that $u^h \to 0$ in $L^2(\Omega)$ we conclude
\begin{equation*}
\lim_{h \to 0}\int_{\Omega}  \nabla_h u^h \cdot \Psi dx = \lim_{h \to 0} \frac{1}{h}\int_{\Omega} u^h \cdot \partial_{x_3}{\Psi_3}dx.
\end{equation*}
 We know that if $\diver_y\Psi = 0$ and $\int_{Y^3} \Psi(x,y) dy$ for every $x \in \Omega$ then  there is a function $\tilde{\Psi} \in C_0^{\infty}(\Omega;C^{\infty}( \mathcal{Y}^3;\R^{3 \times 3}))$ such that $\curl_y\tilde{\Psi}= \Psi$ (this can be proved by Fourier transform). Thus, we obtain that
\begin{align*}
\lim_{h \to 0}\frac{1}{h} \int_{\Omega}  u^h \cdot (\partial_{x_3} \Psi) dx &= \frac{1}{h} \int_{\Omega}  u^h \cdot (\partial_{x_3} (\partial_{y_1}\tilde{\Psi}_2 - \partial_{y_2}\tilde{\Psi}_1 )) \\
&= \frac{1}{h} \int_{\Omega}  u^h \cdot ( \partial_{y_1}\partial_{x_3}\tilde{\Psi}_2 - \partial_{y_2}\partial_{x_3} \tilde{\Psi}_1 )) dx \\
&= h^\alpha  \int_{\Omega}  u^h \cdot \left( \frac{\partial_{y_1}\partial_{x_3}\tilde{\Psi}_2}{h^{\alpha+1}}+ \partial_{x_1}\partial_{x_3}\tilde{\Psi}_2 - \frac{\partial_{y_2}\partial_{x_3} \tilde{\Psi}_1 }{h^{\alpha+1}}  - \partial_{x_2}\partial_{x_3}\tilde{\Psi}_1\right) \\
&- h^{\alpha } \int_{\Omega} u^h \cdot \left( \partial_{x_1}\partial_{x_3}\tilde{\Psi}_2 - \partial_{x_2}\partial_{x_3}\tilde{\Psi}_1 \right) \\
&=h^{\alpha} \int_{\Omega} u^h \cdot \left(\partial_1\partial_{x_3}\tilde{\Psi}_2 - \partial_2{\partial_{x_3}\tilde{\Psi}_1}\right) - h^{\alpha } \int_{\Omega} u^h \cdot \left( \partial_{x_1}\partial_{x_3}\tilde{\Psi}_2 - \partial_{x_2}\partial_{x_3}\tilde{\Psi}_1 \right) \\
&=h^{\alpha} \int_{\Omega} \partial_1{u^h} \cdot \partial_{x_3}\tilde{\Psi}_2
 - h^{\alpha} \int_{\Omega} \partial_2{u^h} \cdot \partial_{x_3}\tilde{\Psi}_1 \\
 &-  h^{\alpha } \int_{\Omega} u^h \cdot \left( \partial_{x_1}\partial_{x_3}\tilde{\Psi}_2 - \partial_{x_2}\partial_{x_3}\tilde{\Psi}_1 \right),
 \end{align*}
which converges to zero as $h \to 0$.
Thus, we obtained that
$$
\iint_{\Omega \times Y^3}  U(x,y) \cdot \Psi(x,y) dy dx =
\lim_{h \to 0} \int_{\Omega} \nabla_h u^h  \cdot \Psi\left(x,\frac{x'}{h^{\alpha+1}},\frac{x_3}{h^{\alpha}}\right)  dx = 0,
$$
for all $\Psi \in C_0^{\infty}(\Omega,C^{\infty}(\mathcal{Y}^3))$ such that $\diver_y \Psi=0$ and $\int_{Y^3} \Psi(x,y) dy=0$ for every $x \in \Omega$. Hence   $U$ is a gradient map with respect to $y$, i.e., there is a function $u^2 \in L^2(\Omega;H^1( \mathcal{Y}^3;\R^3))$ such that
$U=\nabla_y u^2$ (this can be proved by Fourier transform).

\item Again, there is a function $U \in L^2(\Omega\times Y^3;\R^{3 \times 3})$ such that $\int_{Y^3} U(x,y) dy=0$ for a.e. $x \in \Omega$ and $\nabla_h u^h \stackrel{2}{\rightharpoonup} U$, i.e.,
$$
\iint_{\Omega \times Y^3}  U(x,y) \cdot \Psi(x,y) dy dx =
\lim_{h \to 0} \int_{\Omega} \nabla_h u^h  \cdot \Psi\left(x,\frac{x'}{\eps_1(h)},\frac{x_3}{\eps_2(h)}\right)  dx.
$$
We take the test function $\Psi$ independent of $y_3$ variable. Furthermore, we assume that
$$\partial_{y_1}\Psi_1 + \partial_{y_2} \Psi_2 + \partial_{x_3} \Psi_3 =0.$$
We obtain that
\begin{eqnarray*}
\int_{\Omega} \nabla_h u^h \cdot \Psi &=& \int_{\Omega} u^h \cdot ( \partial_{x_1}\Psi_1 + \partial_{x_2}\Psi_2 ) dx\\ & & + \int_{\Omega} u^h \cdot \left( \frac{1}{\eps_1(h)}\partial_{y_1}\Psi_1 + \frac{1}{\eps_1(h)} \partial_{y_2}\Psi_2 + \frac{1}{h} \partial_{x_3}\Psi_3 \right).
\end{eqnarray*}
Taking the limit and using the fact that $u^h \to 0$ strongly in $L^2(\Omega)$ we conclude
$$\int_{\Omega \times Y^2} \left( \int_{Y^1} U(x,y',y_3) dy_3 \right) \Psi(x,y') dy' dx = 0. $$
 From this we conclude, as in \cite[Theorem 6.3.3]{Neukamm-10} that there is a function $u^1 \in L^2(\Omega; H^1(\mathcal{Y}^2;\R^3))$ such that
$$\int_{Y^1} U(x,y',y_3) dy_3 = \left( \nabla_y u^1 \vert \partial_{x_3} u^1 \right).$$
Denote the first two columns of $U$ by $U_1$ and $U_2$. It remains to show that $U_1$ and $U_2$ are independent of $y_3$. To prove this take any test function $\Psi$ and note that
$$\int_{\Omega} \partial_\alpha u^{h} \cdot  \eps_2(h) \partial_{3} \Psi dx  \to \iint_{ \Omega \times Y^3} U_\alpha \cdot \partial_{y_3}{\Psi} dy dx.$$
On the other hand, by partial integration we derive that
\begin{align*}
\eps_2(h)\int_{\Omega} \partial_{\alpha} u^{h} \cdot   \partial_{3}  \Psi dx &=
\frac{h \eps_2(h)}{\eps_1(h)}  \int_{\Omega} \frac{1}{h}\partial_{x_3} u^{h} \cdot   \left( \eps_1(h) \partial_{x_\alpha}  \Psi +  \partial_{y_\alpha}\Psi \right)dx,
\end{align*}
which converges to zero as $h \to 0$. Thus we derived that
$$\iint_{ \Omega \times Y^3} U_\alpha \cdot \partial_3{\Psi} dy dx = 0.$$
From this we see that $U_1$ and $U_2$ are independent of $y_3$ and the claim is proved.

\end{enumerate}
\end{proof}

We define the limit energy density
\begin{definition}{(Relaxation formula)}\label{defq}
We define $Q_0^{p}:\R^{2 \times 2} \to [0,\infty)$ in the following two cases
\begin{enumerate}
\item ($\eps_1(h)=h^{\alpha+1}, \eps_2(h)=h^{\alpha}$)
$$Q_0^{p}(A):=\inf_{B} \inf_{d}\int_{I} \left(\inf_{\phi} \int_{Y^3} Q\left(y,\iota(B+x_3 A) +(0\vert 0 \vert d)+ \nabla_y\phi\right) dy\right) dx_3,$$
where the infimum  is taken over $B \in S_2$, $d \in L^2(I;\R^3)$ and $\phi \in H^1(\mathcal{Y}^3; \R^3)$.
\item ($\lim_{h \to 0} \frac{\eps_1(h)}{h}=1$, $\lim_{h \to 0}\eps_2(h)=0$)
$$Q_0^p(A):=\inf_{B} \inf_{\phi} \int_{I \times Y^2} \left( \inf_{d} \int_{Y^1} Q\left(y,\iota(B+x_3 A) + \left(\nabla'_y \phi\vert  \partial_{x_3} \phi \right) + (0\vert0\vert d)\right)dy_3 \right) dy'dx_3, $$
where the infimum is taken over $B \in S_2$, $\phi \in H^1 (I \times \mathcal{Y}^2; \R^3)$, $d \in L^2(Y^1;\R^3)$ such that $\int_{Y^1} d(y_3) dy_3=0$.
\end{enumerate}
\end{definition}
\begin{proposition}
If $Q:[0,1)^3 \times \R^{3 \times 3} \to [0,\infty) $ is such that $Q(y,\cdot)$ is a quadratic form for a.e. $y  \in [0,1)^3$ that satisfies (\ref{svojstvoq}), then $Q_0^p$ defined in Definition \ref{defq} is a quadratic form that satisfies $Q_0^p(A)=Q_0^p(\sym A)$, for all $A \in \R^{2 \times 2}$. Moreover, $Q_0^p$ is positive definite on symmetric matrices.
\end{proposition}
\begin{proof}
We will only prove the case when $\eps_1(h)=h^{\alpha+1}, \eps_2(h)=h^{\alpha}$. The other case goes in the similar way. Notice that
$Q_0^p$ can be written in the following way
$$Q_0^p(A):=\inf_{B,\phi,d}\int_{I \times Y^3}  Q\left(y,\iota(B+x_3 A) +(0\vert 0 \vert d)+ \nabla_y\phi\right) dy dx_3,$$
where the infimum is taken over $B \in S_2$, $d \in L^2(I;\R^3)$, $\phi \in L^2(I; H^1(\mathcal{Y}^3;\R^3))$. If we denote by $\mathcal{V}$ the closed subspace of $L^2(I \times Y^3;S_3)$ given by
$$ \mathcal{V}=\{\iota(B)+\sym (0\vert 0 \vert d)+ \sym \nabla_y\phi: B \in S_2, d \in L^2(I;\R^3), \phi \in L^2(I; H^1(\mathcal{Y}^3;\R^3))\}, $$
then we can interpret minimizing $(B,\phi,d)$ as the projection of $\iota(\sym (x_3 A))$ on $\mathcal{V}$ in the norm defined by the square root of the integral of the function $Q$. From this we have the uniqueness of the minimizer (notice that the decomposition in the definition of $\mathcal{V}$ is orthogonal with respect to the standard scalar product in $L^2(I \times Y^3;S_3)$) and also the quadraticity of $Q_0^p$. The coercivity of $Q_0^p$ follows from the coercivity of $Q$ and orthogonality of $\sym (x_3 A)$ on the subspace $\mathcal{V}_1$ of
$L^2(I \times \mathcal{Y}^2;S_2)$ defined by
$$ \mathcal{V}_1=\{B+\sym \nabla_y' \phi: B \in S_2, \phi \in  L^2 (I; H^1(Y^3;\R^2))\},$$
in the standard scalar product on $L^2(I \times \mathcal{Y}^2;S_2)$.
\end{proof}

We define the energy functionals by
$$I^{h,\eps_1(h),\eps_2(h)}(u^h) = \int_{\Omega}W\left(\frac{x'}{\eps_1(h)},\frac{x_3}{\eps_2(h)},\nabla_h u^h  \right)dx,$$
and the limit functional by
$$I_0^p(u)=\left\{ \begin{array}{c}
\int_\omega Q_0^p(II(x')) dx'  \quad \mbox{ if }u \in W^{2,2}_{iso} \\
+\infty, \quad \mbox{otherwise}
\end{array}
\right. . $$

By using the Proposition \ref{prop:two-scale} and by applying the same techniques as in \cite{Horneuvel12} we obtain the following result. As standard in $\Gamma$-convergence together with compactness Lemma \ref{lemastrong} it implies the desired convergence of minimizers.
\begin{theorem} Assume that for some $C>0$ we have
$$\liminf_{h \to 0} \frac{1}{h^2}I^{h,\eps_1(h),\eps_2(h)}(u^h) \leq C. $$ Then
\begin{enumerate}[(i)]
\item (Lower bound). If $(u^h)_{h>0}$ is a sequence with $u^h - \fint_{\Omega} u^h \to u$ in $L^2(\Omega;\mathbb{R}^3)$ then
$$\liminf_{h \to 0} \frac{1}{h^2}I^{h,\eps_1(h),\eps_2(h)}(u^h) \geq I_0^p(u).$$

\item (Upper bound). For every $u \in W^{2,2}_{iso}(S,\mathbb{R}^3)$ there is a sequence $(u^h)_{h>0}$ with $u^h - \fint_{\Omega} u^h \to u$ in $L^2(\Omega;\mathbb{R}^3)$ such that
$$\lim_{h \to 0} \frac{1}{h^2}I^{h,\eps_1(h),\eps_2(h)}(u^h) = I_0^p(u)$$
\end{enumerate}

\end{theorem}
\begin{proof}
We will just give the sketch of the proof since the main indegredient comparing to the results in \cite{Horneuvel12} is Proposition \ref{prop:two-scale}.
We again define the sequence $z^{h}$ as in (\ref{seq}).  Next we identify the approximate strain
$$G^{h}=B^{h}+ x_3 A^{h} + (R^{h})^t \nabla_{h} z^{h},$$
where $B^{h} \in L^2(S;\R^{3\times 2})$ and $A^h \in L^2(S;\R^{3\times 2})$ are given by the expressions (\ref{defah}) and (\ref{defbh}).
Since $B^h$ and $A^h$ depend only on $x'$  their two-scale limit depends only on variables $x'$ and $y'$.
Using the calculations in \cite[Proposition 3.2]{Horneuvel12} as well as Proposition \ref{prop:two-scale}  we conclude that on every converging subsequence (not relabeled) is valid
\begin{enumerate}
\item in the case $\eps_1(h)=h^{\alpha+1}, \eps_2(h)=h^{\alpha}$  there exists $B \in L^2(S;S_2)$, $d \in L^2(\Omega;\R^3)$,  $\phi \in L^2(\Omega; H^1(\mathcal{Y}^3;\R^3)$ such that
    $$ \sym G^h \stackrel{2}{\rightharpoonup} \iota(B+x_3II)+(0\vert0\vert d)+\nabla_y \phi.$$
\item in the case $\lim_{h \to 0} \frac{\eps_1(h)}{h}=1$, $\lim_{h \to 0}\eps_2(h)=0$   there exists $B \in L^2(S;S_2)$, $\phi \in L^2(\Omega;H^1(I \times \mathcal{Y}^2;\R^3))$, $d \in L^2(\Omega \times Y^3;\R^3)$ such that $\int_{Y^1} d(x,y',y_3) dy_3=0$ for a.e. $(x,y') \in \Omega \times Y^2$ and
$$ \sym G^h \stackrel{2}{\rightharpoonup}   \iota(B+x_3II)+\left(\nabla'_y \phi\vert  \partial_{x_3} \phi \right) + (0\vert0\vert d) .$$
    \end{enumerate}

 Using lower semicontinuity of the convex functionals (\cite{Visintin-07}, Proposition 1.3)  we can prove the lower bound following the proof of \cite[Theorem 2.4]{Horneuvel12}. The upper bound is obtained directly by using again the density  of $\mathcal{A}(S)$ as in \cite{Horneuvel12}.

\end{proof}
\begin{remark} \label{napomena2} If we analyze the relaxation formula in Definition \ref{defq} we see that in the first case the effective behavior is obtained by minimizing with respect to  the gradient of $\phi \in H^1(\mathcal{Y}^3;\R^3)$ which carries the in-plane oscillations coupled together with the oscillations in $x_3$ direction. Thus we can conclude that in this case we have strong coupled homogenization effects. In the second case we firstly minimize with respect to the oscillations in $x_3$ direction and then with respect to the in-plane oscillations coupled with macroscopic behavior in $x_3$ direction.
Although there is no coupling in the relaxation field between in-plane oscillations and oscillations in $x_3$ direction, homogenization effects still exist as we will see in the next example (in fact relaxing firstly with respect to the oscillations in $x_3$ direction influences the energy density which then has an influence on the macroscopic behavior in $x_3$ direction which is coupled with in-plane oscillations).
\end{remark}
\begin{example}
We assume that we are in the regime $\lim_{h \to 0} \frac{\eps_1(h)}{h}=1$, $\lim_{h \to 0}\eps_2(h)=0$ and that the quadratic form $Q$ is isotropic and the Poisson's ratio equals to zero. We assume the following form
$$ Q(A)=\lambda_1(y',x_3) \lambda_2 (y_3) |\sym A|^2,$$
where $\lambda_1, \lambda_2$ are positive functions.
After a short computation we obtain
\begin{eqnarray*}
&&\inf_{d} \int_{Y^1}  \left|\sym \left( \iota(B+x_3 A) + \left(\nabla'_y \phi\vert  \partial_{x_3} \phi \right) + (0\vert0\vert d) \right)\right|^2 dy_3= \\
&& \hspace{+5ex}\lambda_1(y',x_3) \langle \lambda_2 \rangle \left|\sym \left( \iota(B+x_3 A) +\nabla'_y (\phi_1,\phi_2) \right) \right|^2 \\ && \hspace{+5ex}+
\lambda_1(y',x_3) \frac{1}{\langle  1 / \lambda_2 \rangle} \left(\tfrac{1}{2} (\partial_{y_1} \phi_3+
\partial_{x_3} \phi_1)^2+\tfrac{1}{2} (\partial_{y_2} \phi_3+
\partial_{x_3} \phi_2)^2+(\partial_{x_3} \phi_3)^2\right),
\end{eqnarray*}
where the infimum is taken over $d \in L^2(Y^1;\R^3)$ such that $\int_{Y^1} d=0$ and  we have denoted by $\langle \cdot \rangle$ the integral of the function under the brackets over the interval $Y^1$. Further  calculation (taking the integral over $I \times Y^2$ and minimizing with respect to $B$ and $\phi$) is not easy, but even in this form we see the homogenization effects.
\end{example}


{\bf Acknowledgement.} This work has been fully supported by Croatian Science Foundation grant number 9477.

\bibliographystyle{alpha}

\end{document}